\documentclass[12pt]{amsart}
\usepackage{amsmath} 
\usepackage{amssymb}
\usepackage{amsfonts}

\newcommand{\bc}{\mathbf C}

\newcommand{\br}{\mathbf R}

\newcommand{\Cal}{\mathcal}

\newcommand{\coker}{\operatorname{Coker}}

\renewcommand{\dim}{\operatorname{Dim}}

\newcommand{\id}{\operatorname{Id}}

\newcommand{\im}{\operatorname{Im}}

\renewcommand{\ker}{\operatorname{Ker}}

\newcommand{\mn}[1]{\Vert#1\Vert}

\newcommand{\ol}{\overline}

\newcommand{\ran}{\operatorname{Ran}}

\newcommand{\re}{\operatorname{Re}}

\newcommand{\restr}[1]{\big|_{#1}}

\newcommand{\set}[1]{\left\{\,#1\,\right\}}

\newcommand{\supp}{\operatorname{\rm supp}}

\newcommand{\w}[1]{\langle #1\rangle }

\newcommand{\wt}{\widetilde}

\newcommand{\rb}{{\rm b}}

\numberwithin{equation}{section}

\gdef\theoremheaderfont#1{\gdef\theorem@headerfont{#1}}

\def\skipmath@b#1\(#2\){{#1}
  \ifx\skipmath@b#2\else\(#2\)\expandafter\skipmath@b\fi}

\theoremheaderfont{\sc}
\newtheorem{thm}{Theorem}[section]
\newtheorem{lem}[thm]{Lemma}
\newtheorem{cor}[thm]{Corollary}
\newtheorem{prop}[thm]{Proposition}

\newtheorem{rem}[thm]{Remark}

\theoremstyle{definition}

\newtheorem{defn}[thm]{Definition}

\newtheorem{exe}[thm]{Example}

\theoremstyle{remark}

\begin{document}

\title[Solvability and Subellipticity]
{The Solvability and Subellipticity of Systems\\of Pseudodifferential Operators}
\author[NILS DENCKER]{{\textsc Nils Dencker}}
\dedicatory{Dedicated to Ferruccio Colombini on his sixtieth birthday}
\address{Centre for Mathematical Sciences, Lund University, Box 118,
SE-221 00 Lund, Sweden}
\email{dencker@maths.lth.se}
\date{December 1, 2008}



\maketitle

\baselineskip 18pt 
\lineskip 2pt
\lineskiplimit 2pt

\section{Introduction}

In this paper we shall study the question of solvability and
subellipticity of square systems of classical pseudodifferential
operators of principal type on a $C^\infty$ manifold~$X$. These are
the pseudodifferential operators which have an 
asymptotic expansion in homogeneous terms, where the highest order term,
the principal symbol, vanishes of first order on the kernel.
Local solvability for an $N \times N$ system of pseudodifferential
operators~$P$ at a compact set $K \subseteq X$ means that
the equations
\begin{equation}\label{locsolv}
Pu = v 
\end{equation}
have a local weak solution $u \in \Cal D'(X, \bc^N)$ in a neighborhood of $K$
for all $v\in C^\infty(X, \bc^N)$ in a subset of finite codimension.  
We can also define microlocal solvability at any compactly based cone
$K \subset T^*X$, see ~\cite[Definition~26.4.3]{ho:yellow}. 
Hans Lewy's famous counterexample~\cite{lewy} from 1957 showed that not all
smooth linear partial differential operators are solvable.

In the scalar case, Nirenberg and Treves conjectured in ~\cite{nt} that local solvability
of scalar classical pseudodifferential operators of principal type is
equivalent to condition~(${\Psi}$) on the 
principal symbol ~$p$. Condition~(${\Psi}$) means that
\begin{multline}\label{psicond} \text{$\im (ap)$
    does not change sign from $-$ to $+$}\\ 
 \text{along the oriented
    bicharacteristics of $\re (ap)$}
\end{multline}
for any $0 \ne a \in C^\infty(T^*X)$. These oriented bicharacteristics are
the positive flow-outs of the Hamilton vector field 
$$H_{\re (ap)} = \sum_j \partial_{{\xi}_j}\re (ap) \partial_{x_j} -
\partial_{x_j} \re(ap) \partial_{{\xi}_j} $$ 
on $\re (ap) =0$, and are called semibicharacteristics of ~$p$. 
The Nirenberg-Treves conjecture was recently proved
by the author, see~ \cite{de:NT}. 

Condition~\eqref{psicond} is obviously invariant under symplectic changes of
coordinates and multiplication with non-vanishing factors.
Thus the condition is invariant under conjugation of ~$P$ with elliptic
Fourier integral operators. 
We say that $p$ satisfies condition~ ($\ol {\Psi}$) if $\ol p$
satisfies condition~ (${\Psi}$), which means that only sign changes
from~$-$ to~$+$ is allowed in~\eqref{psicond}. We also say that $p$
satisfies condition ($P$) if 
there are no sign changes on the semibicharacteristics, that is, $p$ satisfies
both condition ~(${\Psi}$) and~ ($\ol {\Psi}$). For partial
differential operators condition  ~(${\Psi}$) and~ ($P$) are
equivalent, since the principal symbol is either odd or even in~${\xi}$.

For systems there is no corresponding conjecture for solvability.
We shall consider systems of principal type, so that the principal
symbol vanishes of first order on the kernel, see Definition~~\ref{princtype}.  
By looking at diagonal operators, one finds that condition ~(${\Psi}$) for
the eigenvalues of the principal symbol is necessary for solvability.
A special case is when we have constant characteristics, so that the
eigenvalue close to the origin has constant 
 multiplicity, see Definition~\ref{constchar}. Then, the eigenvalue is a
$C^\infty$ function and condition~ (${\Psi}$) is well-defined.
For classical systems of pseudodifferential operators
of principal type having eigenvalues of the principal symbol with 
constant multiplicity, the generalization of the Nirenberg-Treves
conjecture is that local solvability is equivalent to condition~
(${\Psi}$) on the eigenvalues. This has recently been proved by the
author, see Theorem~2.7 in~\cite{de:sysolv}. 

But when the principal symbol is not diagonalizable, condition
~(${\Psi}$) is not sufficient for
local solvability, see Example~\ref{saex} below.
In fact, it is not even known if condition~(${\Psi}$) is sufficient
in the case when the principal system is $C^\infty$
diagonalizable. Instead,  we shall study the
{\em quasi-symmetrizable} systems introduced in ~\cite{de:pseudospec}, see Definition~\ref{QS}.
These are of principal type, are invariant under taking adjoints and
multiplication with invertible systems. A scalar quasi-symmetrizable symbol is 
of principal type and satisfies condition~($P$).
Our main result is that quasi-symmetrizable systems are locally solvable, see Theorem~\ref{QSthm}.

We shall also study the subellipticity of square systems. An $N
\times N$ system of
pseudodifferential operators~$P \in
{\Psi}^m_{cl}(X)$ is {\em subelliptic} with a loss of $ 
{\gamma} < 1$ derivatives if 
 $ 
 Pu \in H_{(s)}$ implies that $ u \in H_{(s+m-{\gamma})}
 $ 
locally for $ u \in \Cal D'(X, \bc^N)$. Here $H_{(s)}$ are the standard
$L^2$ Sobolev spaces, thus ellipticity corresponds to
${\gamma} = 0$ so we may assume ${\gamma} > 0$. 
For scalar operators, subellipticity is equivalent to condition~($\ol {\Psi}$) and the bracket
condition on the principal symbol~$p$, i.e., that some repeated Poisson
bracket of $\re p$ and $\im p$ is non-vanishing. This is not true for
systems, and there seems to be no general
results on the subellipticity for systems of pseudodifferential
operators. In fact, the real and imaginary parts
do not commute in general, making the bracket condition meaningless.
Even when they do, the bracket condition is not invariant and 
not sufficient for subellipticity, see Example~\ref{ex1}.

Instead we shall study quasi-symmetrizable symbols, for which we introduce
invariant conditions on the order of vanishing of the symbol  
along the semibicharacteristics of the eigenvalues.
Observe that for systems, there could be several (limit) semibicharacteristics  of the
eigenvalues going through a characteristic point, see Example~\ref{subex}. 
Therefore we introduce the {\em approximation property}
in Definition~\ref{apprdef} which gives that the all (limit)
semibicharacteristics of the eigenvalues are parallell at the characteristics,
see Remark~\ref{condrem}. 
We shall study systems of {\em finite type} introduced in
~\cite{de:pseudospec}, these are 
quasi-symmetrizable systems satisfying the approximation property, for which 
the imaginary part on the kernel vanishes of finite order along the
bicharacteristics of the real part of the eigenvalues. This definition
is invariant under multiplication with invertible systems and taking adjoints. For scalar
symbols this corresponds to the case when the operator satisfies
condition~($P$) and the bracket condition. For system of finite type we obtain
subellipticity with a loss of $2k/2k+1$ derivatives as in the scalar
case, where $2k$ is the order of vanishing, see Theorem~\ref{subthm}.
For the proof, we shall use the estimates developed in ~\cite{de:pseudospec}.
The results in this paper are formulated for operators acting on the
trivial bundle. But since our results are mainly local, they can be applied to
operators on sections of fiber bundles.

\section{Solvability of Systems}

Recall that a scalar symbol $p(x,{\xi})\in C^\infty(T^*X)$ is of {\em principal
  type} if $dp \ne 0$ when $p = 0$. 
We shall generalize this definition to  systems $P\in C^\infty( T^*X)$. 
For ${\nu} \in T_w(T^*X)$, $w = (x,{\xi})$, we let
$\partial_{\nu}P(w) = \w{{\nu}, dP(w)}$. 
We shall denote $\ker P$ the kernel and $\ran P$ the range of the matrix~$P$.

\begin{defn} \label{princtype}
The $N \times N$ system $P(w) \in C^\infty(T^*X)$ is of {\em
  principal type } at $w_0$ if
\begin{equation}\label{pr_type}
\ker P(w_0) \ni u \mapsto \partial_{\nu}P(w_0)u \in \coker P(w_0)  = \bc^N/\ran P(w_0)
\end{equation} 
is bijective for some ${\nu} \in T_{w_0}(T^*X)$. The operator $P \in {\Psi}_{cl}^m(X)$
is of principal type if the homogeneous principal symbol
${\sigma}(P)$ is of principal type. 
\end{defn}

Observe that if $P$ is homogeneous in~${\xi}$, then the direction~${\nu}$
cannot be radial. In fact, if ~${\nu}$ has the radial direction
and~$P$ is homogeneous then 
$ \partial_{\nu}P = cP$
which vanishes on $\ker P$.

\begin{rem}\label{princrem}
If $P(w) \in C^\infty$ is of principal type and $A(w)$, $B(w) \in
C^\infty$ are invertible then  $APB$ is of principal
type. We have that $P$ is of principal type if and only if the adjoint $P^*$
is of principal type. 
\end{rem}

In fact, by Leibniz' rule we have
\begin{equation} \label{dapb}
 \partial (APB) = (\partial A)PB + A(\partial P)B+  AP \partial B
\end{equation}
and $\ran (APB) = A (\ran P)$ and $\ker (APB) =
B^{-1}(\ker P)$ when $A$ and $B$ are invertible, which gives
invariance under left and right multiplication.  
Since $ \ker P^*(w_0) = \ran
P(w_0)^\bot$ we find that $P$ satisfies ~\eqref{pr_type} if and 
only if  
\begin{equation}\label{bilform}
 \ker P(w_0) \times \ker P^*(w_0) \ni (u,v) \mapsto
\w{\partial_{\nu}P(w_0) u,v}
\end{equation}
is a non-degenerate bilinear form. 
Since $\w{\partial_{\nu}P^* v,u} = \ol{\w{\partial_{\nu}P u, v}}$ 
we then obtain that $P^*$ is of principal type.

Observe that if $P$ only has one vanishing eigenvalue ${\lambda}$ (with
multiplicity one) then the condition that $P$ ~is of
principal type reduces to the condition in the scalar case:
$d{\lambda} \ne 0$ when ${\lambda}= 0$.
In fact, by using the
spectral projection one can find invertible systems~ $A$ and~
$B$ so that
\begin{equation*}
 APB = 
\begin{pmatrix}
{\lambda} & 0\\ 0 & E 
\end{pmatrix} \in C^\infty
\end{equation*}
where $E$ is an invertible $(N-1) \times (N-1)$ system. Since this system is
of principal type we obtain the result by the invariance.

\begin{exe}
Consider the system
\begin{equation*}
 P(w) = 
\begin{pmatrix}
{\lambda}_1(w) & 1 \\ 0 & {\lambda}_2(w) 
\end{pmatrix}
\end{equation*}
where ${\lambda}_j(w) \in C^\infty$, $j=1$, 2. Then $
P(w)$ is not of principal type 
when ${\lambda}_1(w) = {\lambda}_2(w) = 0$ since then 
$\ker P(w) =
\ran P(w) = \bc \times\set{0}$, which is preserved by $\partial P$.
\end{exe}

Observe that the property of being of principal type is not stable
under $C^1$ perturbation,
not even when $P = P^*$ is symmetric by the following example.

\begin{exe}\label{prtrem} 
The system
\begin{equation*}
 P(w) = 
\begin{pmatrix}
w_1 - w_2 & w_2 \\ w_2 & -w_1 -w_2  
\end{pmatrix} = P^*(w) \qquad w = (w_1,w_2)
\end{equation*}
is of principal type when $w_1 = w_2 = 0$, but {\em not} of principal
type when $w_2 \ne 0$ and $w_1 = 0$. In fact, 
\begin{equation*}
 \partial_{w_1}P =
\begin{pmatrix} 1 & 0 \\ 0 & -1 
\end{pmatrix}
\end{equation*}
is invertible, and when $w_2 \ne 0$ we have that 
$$\ker P(0,w_2)= \ker  \partial_{w_2}P(0,w_2) =\set{z(1,1):
  z \in \bc}$$ 
which is mapped to\/ $\ran P(0,w_2)  =\set{z(1,-1):  z \in \bc}$ by
$\partial_{w_1}P$.
The eigenvalues of ~$P(w)$ are
$ -w_2 \pm \sqrt{w_1^2 + w_2^2}$ which are equal if and only if $w_1 =
w_2 = 0$. 
When $w_2 \ne 0$ the eigenvalue close to zero is $w_1^2/2w_2 + \Cal
O(w_1^4)$ which has vanishing differential at~$w_1 = 0$. 
\end{exe}

Recall that the multiplicity of  ${\lambda}$ as a root of the characteristic equation $|P(w) -
{\lambda}\id_N| = 0$ is the {\em algebraic}  multiplicity of the eigenvalue,
and the dimension of $\ker (P(w) - {\lambda}\id_N)$ is the {\em
geometric} multiplicity. 
Observe the geometric multiplicity is lower or equal to the
algebraic, and for symmetric systems they are equal.

\begin{rem} \label{smoothev}
If the eigenvalue ${\lambda}(w)$ has
constant {\em algebraic} multiplicity then it is a $C^\infty$
function.  
\end{rem}

In fact, if $k$ is the multiplicity then ${\lambda} =
{\lambda}(w)$ solves $\partial_{\lambda}^{k-1}|P(w)
- {\lambda}\id_N| = 0$ so we obtain this from
the Implicit Function Theorem. 
This is { not} true when we have constant geometric multiplicity, for
example $P(t) = 
\begin{pmatrix} 
0 & 1 \\ t & 0 
\end{pmatrix}$, $t \in \br$, has geometric multiplicity equal to one for the
eigenvalues $\pm\sqrt{t}$.

Observe that if the matrix $P(w)$ depend continuously on a parameter
$w$, then the eigenvalues ${\lambda}(w)$ also depend continuously on ~$w$. 
Such a continuous function~${\lambda}(w)$ of eigenvalues we
will call a {\em section of eigenvalues of\/}~$P(w)$.

\begin{defn} \label{constchar}
The $N \times N$ system $P(w) \in C^\infty$ has {\em
constant characteristics} near $w_0$ if there exists an ${\varepsilon}
> 0$ such that any section of eigenvalues~${\lambda}(w)$ of~$P(w)$
with $|{\lambda}(w)| < {\varepsilon}$ has both constant algebraic and constant
geometric multiplicity in a neighborhood of ~$w_0$.
\end{defn}

If $P$ has constant characteristics then the section of  eigenvalues close
to zero has constant algebraic multiplicity, thus it is a
$C^\infty$ function close to zero.
We obtain from Proposition~2.10 in ~\cite{de:sysolv} that
if $P(w) \in C^\infty$ is an $N \times N$ system of constant
characteristics near $w_0$, then $P(w)$ is of principal type at~ $w_0$
if and only 
if the algebraic and geometric multiplicities of $P$ agree at~$w_0$ and
$d{\lambda}(w_0) \ne 0$ for the $C^\infty$ section of
eigenvalues~${\lambda}(w)$ for ~$P$ satisfying 
${\lambda}(w_0) = 0$, thus there are
no non-trivial Jordan boxes in the normal form.

For classical systems of pseudodifferential operators of principal type and
constant characteristics, the  eigenvalues are homogeneous $C^\infty$ functions
when the values are close to zero, so the condition~(${\Psi}$) given
by ~\eqref{psicond} is well-defined on the eigenvalues. 
Then, the natural generalization of the Nirenberg-Treves
conjecture is that local solvability is equivalent to condition~
(${\Psi}$) on the eigenvalues. This has recently been proved by the
author, see Theorem~2.7 in~\cite{de:sysolv}.

When the multiplicity of the eigenvalues of the principal symbol is
not constant the situation is much  more complicated.
The following example shows that then it is not
sufficient to have conditions only on the eigenvalues
in order to obtain solvability, not even 
in the principal type case.

\begin{exe} \label{saex}
Let $x \in \br^2$, $D_x = \frac{1}{i}\partial_x$ and
$$
P(x,D_x) =
\begin{pmatrix}
D_{x_1}  & x_1D_{x_2} \\ x_1D_{x_2} & - D_{x_1}
\end{pmatrix} = P^*(x,D_x)
$$ 
This system is symmetric of principal type and  ${\sigma}(P)$ has real
eigenvalues $\pm \sqrt{{\xi}_1^2 +
  x_1^2{\xi}_2^2}$ but
$$\frac{1}{2}
\begin{pmatrix}
1 & -i \\ 1 & i 
\end{pmatrix}P
\begin{pmatrix}
1 & 1 \\ -i & i 
\end{pmatrix}  
=
\begin{pmatrix}
D_{x_1} - ix_1D_{x_2} &  0 \\ 0 & D_{x_1} +ix_1D_{x_2}
\end{pmatrix}
$$
which is not solvable at $(0,0)$ because condition (${\Psi}$) is not satisfied.
The eigenvalues of the principal symbol are now ${\xi}_1 \pm
ix_1{\xi}_2$.
\end{exe}

Of course, the problem is that the eigenvalues are not invariant under
multiplication with elliptic systems.
We shall instead study
{\em quasi-symmetrizable} systems, which generalize
the normal forms of the scalar symbol at the boundary of the
numerical range of the principal symbol, see Example~\ref{scalarcase}.

\begin{defn} \label{QS} The $N \times N$ system  $P(w) \in
C^\infty(T^*X)$ is {\em quasi-symmetrizable}  with respect to a
{real} $C^\infty$ vector field  
$V$ in ${\Omega} \subseteq T^*X$ if $\exists\ N\times N$ system $ M(w)
\in C^\infty(T^*X)$ so that
\begin{align}
&\re \w{M(VP)u,u} \ge c\mn u^2 - C\mn{Pu}^2 \qquad c >
0   \qquad \forall\, u  \in \bc^N\label{qs1}\\
&\im \w{MPu,u} \ge  - C\mn{Pu}^2 \qquad \forall\, u  \in \bc^N\label{qs2}
\end{align}
on ${\Omega}$, the system 
$M$ is called a {\em symmetrizer} for $P$.
If $P \in {\Psi}^m_{cl}(X)$ then it is quasi-symmetrizable if
the homogeneous principal symbol~${\sigma}(P)$ is quasi-symmetrizable
when $|{\xi}|=1$, one can then choose a homogeneous symmetrizer~$M$.
\end{defn}

The definition is clearly independent
of the choice of coordinates in $T^*X$ and choice of basis in
$\bc^N$.
When $P$ is elliptic, we find that $P$
is quasi-symmetrizable with respect to any vector field since $\mn{Pu}
\cong \mn{u}$.
Observe that the set of
symmetrizers ~~$M$ satisfying~~\eqref{qs1}--\eqref{qs2} is a
convex cone, a sum of two 
multipliers is also a multiplier. Thus for a given vector field ~$V$ it
suffices to make a local choice of symmetrizer and then use a partition
of unity to get a global one.

\begin{exe}\label{scalarcase}
A scalar function $p \in C^\infty$ is quasi-symmetrizable
if and only
\begin{equation}\label{normform}
  p(w) = e(w)(w_1 + if(w')) \qquad w= (w_1,w')
\end{equation} 
for some choice of coordinates, where $f \ge 0$. Then $0$ is at the boundary of the 
numerical range of~$p$.
\end{exe}

In fact, it is obvious that $p$ in~\eqref{normform} is  quasi-symmetrizable. 
On the other hand, if $p$ is quasi-symmetrizable then there exists $m
\in C^\infty$ such that $mp = p_1 + ip_2$ where $p_j$ are real satisfying
$\partial_{\nu} p_1 > 0$ and $p_2 \ge 0$. Thus $0$ is at the boundary of the
numerical range of ~$p$. By using Malgrange
preparation theorem and changing coordinates as in the proof of Lemma~4.1 in ~\cite{dsz}, 
we obtain the normal form~\eqref{normform} with  $\pm f \ge 0$.

Taylor has studied  {\em symmetrizable} systems
of the type $D_t\id + i K$, for which there exists $R > 0$
making $RK$ symmetric (see Definition~4.3.2 in \cite{ta:pseu}). 
These systems are quasi-symmetrizable with respect to $\partial_{\tau}$ with
symmetrizer~$R$.
We shall denote $\re A =
\frac{1}{2}(A + A^*)$ and $i\im A = \frac{1}{2}(A - A^*)$ the symmetric
and antisymmetric parts of the matrix ~$A$.
Next, we recall the following result from Proposition~4.7  in ~\cite{de:pseudospec}.

\begin{rem}\label{princlem}
If the  $N \times N$ system $P(w) \in C^\infty$ is
quasi-symmetrizable then it is of principal type.
Also, the symmetrizer $M$ is invertible if $\im MP \ge c P^*P$
for some $c> 0$.
\end{rem}

Observe that by adding $i{\varrho}P^*$ to $M$ we may assume that $Q = MP$ satisfies
\begin{equation} \label{immp}
\im Q \ge
({\varrho}-C)P^*P \ge P^*P \ge c Q^*Q \qquad c > 0
\end{equation} 
for ${\varrho} \ge C+1$, and then the symmetrizer is invertible by
Remark~\ref{princlem}.

\begin{rem}\label{eqrem}
The system $P\in C^\infty$ is quasi-symmetrizable with respect to $V$
if and only if there exists an invertible symmetrizer~$M$ such that $Q = MP$ satisfies 
\begin{align}
&\re \w{ (V Q) u,u} \ge c\mn u^2 - C\mn{Qu}^2   \qquad c > 0 \label{qs1b}\\
&\im \w{Q u,u} \ge 0 \label{qs2b}
\end{align}
for any $u \in \bc^N$.
\end{rem}

In fact, by the Cauchy-Schwarz inequality we find
\begin{equation*}
|\w{(V M)Pu,u}| \le {\varepsilon}\mn u^2 +
C_{\varepsilon}\mn{Pu}^2 
 \qquad \forall\, {\varepsilon} >0 \quad \forall\, u \in \bc^N
\end{equation*}
Since $M$ is invertible, we also have that $\mn {Pu} \cong \mn{Qu}$.

\begin{defn}
If $Q\in C^\infty(T^*X)$ satisfies ~\eqref{qs1b}--\eqref{qs2b} then $Q$
is {\em  quasi-symmetric} with respect to the real $C^\infty$ vector field~$V$.
\end{defn}

The invariance properties of quasi-symmetrizable systems is partly due to the
following properties of semibounded matrices. Let $U + V = \set{u + v:\ u \in U\  \land\  v \in V}$ for
linear subspaces $U$ and $V$ of $\bc^N$.

\begin{lem}\label{semiprop}
Assume that $Q$ is an $N \times N$ matrix such that $\im zQ \ge 0$
for some $0 \ne z \in \bc$. Then we find 
\begin{equation}\label{kereq}
 \ker Q = \ker Q^* = \ker (\re Q) \bigcap \ker (\im Q)
\end{equation}
and $\ran Q = \ran (\re Q) + \ran (\im Q) \bot \ker Q$.
\end{lem}

\begin{proof}
By multiplying with $z$ we may assume that $\im Q \ge 0$, clearly the
conclusions are invariant under multiplication with complex numbers. If
$u \in \ker Q$, then we have  $ \w{\im Qu,u} = \im\w{Qu,u}= 0$.  By
using the Cauchy-Schwarz inequality on $\im Q \ge 0$ we find that 
$\w{\im Qu,v} = 0$ for any~$v$. Thus $u \in \ker (\im Q)$ so
$\ker Q \subseteq \ker Q^*$. We get equality and ~\eqref{kereq} by the
rank theorem, since $\ker Q^* = \ran Q^\bot$. 

For the last statement we observe that $\ran Q \subseteq \ran (\re Q)
+ \ran (\im Q) = (\ker Q)^{\bot}$ by~\eqref{kereq} where we also get equality by the
rank theorem.
\end{proof}

\begin{prop}\label{invrem0}
If $Q  \in C^\infty(T^*X)$ is quasi-symmetric and ~$E \in C^\infty(T^*X)$ is
invertible, then~ $E^*QE$ and~$-Q^*$ are quasi-symmetric. 
\end{prop}

\begin{proof}
First we note that~\eqref{qs1b} holds if and only if 
\begin{equation}\label{qs1a}
 \re \w{(V Q)u,u} \ge c\mn u^2\qquad \forall\,u \in \ker Q
\end{equation}
for some $c >0 $. In fact, $Q^*Q$ has a positive lower bound on  
the orthogonal complement $\ker Q^{\bot}$ so that 
\begin{equation*}
 \mn {u} \le C\mn{Qu} \qquad \text{for $u \in \ker Q^{\bot}$}
\end{equation*}
Thus, if $u = u' + u''$ with $u' \in \ker
Q$ and $u'' \in \ker Q^{\bot}$ we find that $Qu = Qu''$,
\begin{equation*}
 \re \w{(V Q) u',u''} \ge -{\varepsilon}\mn {u'}^2 - C_{\varepsilon}\mn
 {u''}^2 \ge  -{\varepsilon}\mn {u'}^2 - C'_{\varepsilon}\mn {Qu}^2
 \qquad \forall\, {\varepsilon} > 0 
\end{equation*}
and $\re \w{(V  Q)u'',u''} \ge -C \mn{u''}^2 \ge -C'\mn{Qu}^2$. By
choosing ${\varepsilon}$ small enough we obtain ~\eqref{qs1b}
by using ~\eqref{qs1a} on $u'$.

Next, we note that $\im Q^* = -\im Q$ and $\re Q^* = \re Q$, so~$-Q^*$ satifies
~\eqref{qs2b} and~\eqref{qs1a} with $V$ replaced by $-V$, and thus it is quasi-symmetric.
Finally, we shall show that $Q_E = E^*QE$ is quasi-symmetric when $E$
is invertible. We obtain from ~\eqref{qs2b} that 
$$\im \w{Q_E u,u} = \im \w{Q Eu,Eu} \ge 0 
\qquad \forall\ u  \in \bc^N 
$$
Next, we shall show that $Q_E$ satisfies~\eqref{qs1a} on $\ker Q_E = E^{-1}\ker
Q$, which will give~\eqref{qs1b}. We find from
Leibniz' rule that $ V  Q_E = (V  E^*)Q E + E^* (V  Q) E
+ E^*Q (V  E)$ where ~\eqref{qs1a} gives
\begin{equation*}
 \re \w{E^*( V  Q)Eu,u} \ge c\mn{Eu}^2 \ge c'\mn u^2 \qquad u \in
 \ker Q_E\quad c'> 0
\end{equation*}
since then $Eu \in \ker Q$. Similarly we obtain that $\w{( V  E^*)Q Eu,u} =
0$ when $u \in \ker Q_E$. Now since $\im Q_E \ge 0$ we find from 
Lemma~\ref{semiprop} that
\begin{equation}\label{qstarprop}  
\ker Q^*_E = \ker Q_E
\end{equation}
which gives
$ 
\w{E^*Q ( V  E)u,u} = \w{E^{-1}( V  E)u, Q^*_Eu} =0
$
when $u \in \ker Q_E = \ker Q^*_E$. Thus $Q_E$ satisfies ~\eqref{qs1a}
so it is quasi-symmetric, which finishes the proof.
\end{proof}

\begin{prop}\label{invprop}
Let  $P(w) \in C^\infty(T^*X)$ be a quasi-symmetrizable $N \times
N$ system, then $P^*$ is quasi-symmetrizable. If  $A(w)$ and $B(w) \in
C^\infty(T^*X)$ are  invertible $N \times N$ systems then $BPA$ is
quasi-symmetrizable.
\end{prop}

\begin{proof}
Clearly ~\eqref{qs1b}--\eqref{qs2b} are invariant under
{\em left} multiplication of $P$ with invertible systems~$E$,
just replace~$M$ with $ME^{-1}$.
Since we may write $BPA = B(A^*)^{-1}A^*PA$ it suffices to show that
$E^*PE$ is quasi-symmetrizable if ~$E$ is invertible.
By Remark~\ref{eqrem} there exists a symmetrizer~$M$ so that 
$Q = MP$ is quasi-symmetric, i.e., satisfies ~\eqref{qs1b}--\eqref{qs2b}.
It then follows from Proposition~\ref{invrem0} that
$$Q_E = E^*QE = E^*M(E^*)^{-1}E^*PE$$
is quasi-symmetric, thus $E^*PE$ is
quasi-symmetrizable. 

Finally, we shall prove that $P^*$ is quasi-symmetrizable if $P$ is.
Since $Q = MP$ is quasi-symmetric, we find from
Proposition~\ref{invrem0} that $Q^* = P^*M^*$ is quasi-symmetric. By
multiplying with $(M^*)^{-1}$ from right, we find from the first part of
the proof that $P^*$ is quasi-symmetrizable.
\end{proof}

For scalar symbols of principal type, we find from the normal form in
Example~\ref{scalarcase} that $0$ is on the boundary of the local numerical
range of the principal symbol. This need not be the case for systems
by the following example.  

\begin{exe}
Let
\begin{equation*}
 P(w) = 
\begin{pmatrix}
w_2+ i w_3 & w_1 \\ w_1 & w_2-i w_3  
\end{pmatrix}
\end{equation*}
which is quasi-symmetrizable with respect to $\partial_{w_1}$
with symmetrizer $M = 
\begin{pmatrix}
0 & 1 \\ 1 & 0 
\end{pmatrix}$.
In fact, $\partial_{w_1}MP = \id_2$ and
\begin{equation*}
 MP(w) = 
\begin{pmatrix}
w_1 & w_2 - i w_3 \\ w_2 + i w_3 & w_1 
\end{pmatrix} =  (MP(w))^*
\end{equation*}
so $\im MP \equiv 0$.
Since eigenvalues of $P(w)$ are  $w_2 \pm \sqrt{w_1^2 - w_3^2}$ we
find that $0$ is not a boundary point of the local numerical
range of the eigenvalues. 
\end{exe}

For quasi-symmetrizable systems we have the following 
semiglobal solvability result.

\begin{thm}\label{QSthm}
Assume that $P\in {\Psi}^m_{cl}(X)$ is an $N \times N$ system
and that there exists a real valued function
$T(w) \in C^\infty(T^*X)$ such that $P $ is quasi-symmetrizable 
with respect to the Hamilton vector field $H_T(w)$ in a neighborhood
of a compactly based cone $K\subset T^*X$. Then $P$ is locally solvable
at~ $K$.
\end{thm}

The cone $K\subset T^*X$ is compactly based if
$K\bigcap \set{(x,{\xi}):\ |{\xi}| = 1}$ is compact.
We also get the following local result:

\begin{cor}
Let $P\in {\Psi}_{cl}^m(X)$ be an $N \times N$ system
that is is quasi-symmetrizable at~$w_0 \in T^*X$.
Then $P$ is locally solvable at~$w_0$.
\end{cor}

This follows since we can always choose a function $T$ such
that $V  = H_T$ at ~$w_0$.
Recall that a semibicharacteristic of ${\lambda}\in C^\infty$ is a
bicharacteristic of $\re (a{\lambda})$ for some $0 \ne a \in C^\infty$.

\begin{rem} \label{qshamrem}
If $Q$ is quasi-symmetric with respect to $H_T$ then the limit set at
the characteristics of the non-trivial semibicharacteristics of the
eigenvalues close to zero of ~$Q$ is a union of curves on which $T$ is
strictly monotone, thus they cannot form closed orbits.
\end{rem}

In fact, we have that an eigenvalue 
${\lambda}(w)$ is $C^\infty$ almost everywhere. 
The Hamilton vector field $H_{\re z\lambda}$ then  gives the
semibicharacteristics of~${\lambda}$, and that is determined by
$\w{dQ u,u}$ with $0 \ne u \in
\ker(P - {\lambda}\id_N)$ by the invariance property given by~\eqref{dapb}. Now
$\re \w{(H_TQ)u,u} > 0$ and
$\im d\w{ Qu,u} = 0$ for $u \in \ker P$ by  ~\eqref{qs1b}--\eqref{qs2b}.
Thus by picking subsequences when ${\lambda} \to 0$ we find that the limits of non-trivial
semibicharacteristics of the eigenvalues close to zero give curves on  
which $T$ is strictly monotone, since $H_T{\lambda} \ne 0$.

\begin{exe}\label{sex}
Let  
\begin{equation*}
 P(t,x;{\tau},{\xi}) = {\tau}M(t,x,{\xi})+ iF(t,x,{\xi}) \in
 S^1_{cl}
\end{equation*}
where  $M  \ge c_0 > 0$ and $F \ge 0$. Then $P$ is quasi-symmetrizable with respect to $\partial
_{{\tau}}$ with symmetrizer $\id_N$, so Theorem~\ref{QSthm} gives that
$P(t,x,D_t,D_x)$ is locally solvable.
\end{exe}

\begin{proof}[Proof of Theorem~\ref{QSthm}]
We shall modify the proof of Theorem~4.15 in ~\cite{de:pseudospec},
and derive estimates for the $L^2$ adjoint ~$P^*$ which will give solvability. By
Proposition~\ref{invprop} we find that $P^*$ is quasi-symmetrizable in
~$K$. By the invariance of the conditions, we may multiply with an elliptic
scalar operator to obtain that $P^* \in {\Psi}^1_{cl}$.
By the assumptions, Definition~\ref{QS} and~\eqref{immp}, we find
that there exists a real valued function $T(w) \in C^\infty$ and a
symmetrizer $M (w) \in C^\infty$ so that
$ Q = MP^*$ satisfies
\begin{align}
&\re H_T Q \ge c  - C_0 Q^*Q  \ge c-C_1\im Q \label{9}\\
&\im Q \ge c\, Q^*Q  \ge 0\label{10}
\end{align}
when $|{\xi}| = 1$ near $K$
for some $c > 0$, and we find that $M$ is invertible by
Remark~\ref{princlem}.
Extending by homogeneity, we may assume that ~$M$ and~~$T$ are
homogeneous  of degree ~0 in~${\xi}$, then~$T \in  S^{0}_{1,0}$ and
~$Q \in S^{1}_{1,0}$. Let
\begin{equation} 
M(x,D)P^*(x, D) = Q(x, D)\in {\Psi}^{1}_{cl} 
\end{equation}
which has principal symbol $Q(x,{\xi})$.
Leibniz' rule gives that $\exp(\pm{\gamma}T) \in
 S^0_{1,0}$ for any ${\gamma} > 0$, so we
can define
\begin{equation*}
Q_{\gamma}(x, D) = 
\exp(-{\gamma}T)(x, D)Q(x, D)\exp({\gamma}T)(x, D) \in {\Psi}^{1}_{cl}
\end{equation*}
Since $T$ is a scalar function, we obtain that the symbol of
\begin{equation}\label{11}
 \im Q_\gamma = Q_1 + {\gamma}Q_0 \qquad\text{modulo $S^{-1}$ near $K$}
\end{equation}
where $0 \le Q_1 = \im Q \in S^1$ and $Q_0 \in S^0$ satisfies
\begin{equation}\label{q1est}
Q_0 = \re H_T Q \ge c - C|{\xi}|^{-1}Q_1 \qquad\text{near $K$}
\end{equation}
by  \eqref{9}, \eqref{10} and homogeneity.

Now take $0 \le {\phi} \in S^0_{1,0}$ such that ${\phi}= 1$ near ~$K$
and ${\phi}$ is supported where~ \eqref{9} and~\eqref{10} hold.
If ${\chi} = {\phi}^2$ then we obtain from ~ \eqref{q1est} and the
sharp G\aa rding inequality~\cite[Theorem~18.6.14]{ho:yellow} that
\begin{equation*}
 Q_0(x, D) \ge c_0 {\chi}(x, D) - C\w{D}^{-1}Q_1(x, D) + R(x, D) + S(x,D)
\end{equation*}
where $c_0 > 0$, $R \in S^{-1}$ and $S \in S^0$ with $\supp S \bigcap K = \emptyset$. Thus we obtain
\begin{equation}\label{12}
\im Q_{\gamma}(x, D) \ge c_0{\gamma}{\chi}(x, D)  + (1 +
{\varrho}_{\gamma})Q_1(x, D) + R_{\gamma}(x, D) + S_{\gamma}(x, D)
\end{equation}
where $R_{\gamma} \in S^{-1}$, ${\varrho}_{\gamma} = -{\gamma}C\w{D}^{-1} \in
{\Psi}^{-1}$ and $S_{\gamma} \in S^0$ with $\supp S_{\gamma} \bigcap K
= \emptyset$. The calculus gives that  ${\chi}(x, D) \cong 
{\phi}(x, D){\phi}(x, D)$ modulo ${\Psi}^{-1}$ and
\begin{equation*}
 (1 + {\varrho}_{\gamma})Q_1(x, D) = (1 +
 {\varrho}_{\gamma}/2))Q_1(x, D)(1 +
 {\varrho}_{\gamma}/2)\qquad\text{modulo ${\Psi}^{-1}$} 
\end{equation*}
By using the sharp G\aa rding inequality
we obtain that $Q_1(x, D) \ge R_0(x, D)$ for some $R_0 \in S^0_{1,0}$. Thus we find
\begin{equation*}
  (1 + {\varrho}_{\gamma})Q_1(x, D) \ge (1 +
  {\varrho}_{\gamma}/2)R_0(x, D) (1 + {\varrho}_{\gamma}/2) = R_0(x, D) \ge -C_0
\end{equation*}
modulo terms in ${\Psi}^{-1}$ (depending on ${\gamma}$). Combining
this with~\eqref{12} and using that $\supp(1 -{\phi}) \bigcap K =
\emptyset$, we find for large enough~${\gamma}$ that 
\begin{equation}\label{qsolvest}
 c_1{\gamma}\mn{{\phi}(x, D)u}^2 \le \im\w{Q_{\gamma}(x,D)u,u} +
 \w{A_{\gamma}(x, D) u,u} + \w{B_{\gamma}(x, D) u,u} \qquad u \in C^\infty_0
\end{equation}
where $c_1 > 0$, $A_{\gamma} \in S^{-1}$ and $B_{\gamma} \in S^0$ with
$\supp B_{\gamma} \bigcap K = \emptyset$. 
Next, we fix ${\gamma}$ and apply this to $\exp\left(-{\gamma}T\right)(x, D) u$.
We find by the calculus that 
\begin{equation*}
 \mn{{\phi}(x, D)u} \le
 C(\mn{{\phi}(x, D)\exp\left(-{\gamma}T\right)(x, D) u} +
 \mn{u}_{(-1)}) \qquad u \in  C^\infty_0
\end{equation*}
We also obtain from the calculus that
$$\exp({\gamma}T)(x, D) \exp({-\gamma}T)(x, D) = 1 + r(x,D)$$
with ~$r \in S^{-1}$, which gives
\begin{multline*}
 Q_{\gamma}(x,D)\exp\left(-{\gamma}T\right)(x, D)  =
 \exp\left(-{\gamma}T\right)(x, D) ( 1 + r(x,D))Q(x, D) \\ + \exp\left(-{\gamma}T\right)(x,
 D)[Q(x,D),r(x,D)] 
\end{multline*}
where $[Q(x,D),r(x,D)] \in {\Psi}^{-1}$.
Since $Q(x, D) = M(x, D)P^*(x, D)$ we find 
$$|\w{ \exp\left(-{\gamma}T\right)(x, D) ( 1 + r(x,D)) Q(x, D)v,\exp(-{\gamma}T)(x,
  D) u}| \le C\mn{P^*(x, D)u} \mn{u}$$
Since $ \mn u \le \mn{{\phi}(x, D)u} + \mn{(1-{\phi}(x, D))u}$ and
${\phi}= 1$ near ~$K$ we obtain that
\begin{equation*}
 \mn{u} \le C\left(\mn{P^*(x,D)u} +  \mn{Q(x,D)u} + \mn u_{(-1)}
 \right) \qquad  u \in C^\infty_0
\end{equation*}
where $Q \in S^0$ with
$\supp Q \bigcap K = \emptyset$. We then obtain the local solvability by 
standard arguments.
\end{proof}

\section{Subellipticity of Systems}

\label{subell}

We shall consider the question when a quasi-symmetrizable system is subelliptic.
Recall that an $N \times N$ system of operators~$P \in
{\Psi}^m_{cl}(X)$ is {\em (micro)subelliptic} with a loss of $ 
{\gamma} < 1$ derivatives at $w_0$ if 
\begin{equation*}
 Pu \in H_{(s)} \text{ at $w_0$}\implies u \in H_{(s+m-{\gamma})} \text{ at $w_0$}
\end{equation*}
for $u \in \Cal D'(X, \bc^N)$. Here $H_{(s)}$ is the standard
Sobolev space of distributions $u$ such that $\w{D}^su \in L^2$. 
We say that $u \in H_{(s)}$ microlocally at~$w_0$ if there exists
$a\in S^0_{1,0}$ such that $a \ne 0$ in a conical neighborhood
of~$w_0$ and $a(x,D)u \in H_{(s)}$. Of
course, ellipticity corresponds to 
${\gamma} = 0$ so we shall assume ${\gamma} > 0$.

\begin{exe} \label{scalarex}
Consider the scalar operator 
\begin{equation*}
 D_t + i f(t, x,D_x)
\end{equation*}
with $0 \le f \in C^\infty(\br, S^1_{cl})$, $(t,x) \in \br\times \br^n$,
then we obtain from Proposition~27.3.1 in ~\cite{ho:yellow} that this operator is
subelliptic with a loss of $k/k+1$ derivatives microlocally near $\set{{\tau}=0}$ if and only if
\begin{equation} \label{scalarcond}
 \sum_{j\le k} |\partial_{t}^jf(t,x,{\xi})| \ne 0 \qquad \forall\, x\ {\xi}
\end{equation} 
where we can choose $k$ even.
\end{exe}

The following example
shows that condition~\eqref{scalarcond} is not sufficient for systems.

\begin{exe}\label{ex1}
Let $P = D_t\id_2 + i F(t)|D_x|$ where
\begin{equation*}
 F(t) =
\begin{pmatrix}
t^2 & t^3 \\ t^3 & t^4 
\end{pmatrix} \ge 0
\end{equation*}
Then we have
$F^{(3)}(0)  = 
\begin{pmatrix}
0 & 6 \\ 6 & 0 
\end{pmatrix}
$ which gives that
\begin{equation}\label{noninvcond}
 \bigcap_{j \le 3} \ker F^{(j)}(0) = \set{0}
\end{equation}
But 
\begin{equation*}
 F(t) = 
\begin{pmatrix}
 1 &  t \\ -t & 1
\end{pmatrix}
\begin{pmatrix}
t^2 & 0 \\ 0 & 0 
\end{pmatrix}
\begin{pmatrix}
 1 &  -t \\ t & 1
\end{pmatrix} 
\end{equation*}
so we find
$$
P = (1 + t^2)^{-1}
\begin{pmatrix}
 1 &  t \\ -t & 1
\end{pmatrix}
\begin{pmatrix}
D_t + i(t^2 + t^4)|D_x| & 0 \\ 0 & D_t 
\end{pmatrix}
\begin{pmatrix}
 1 &  -t \\ t & 1
\end{pmatrix} 
\qquad\text{modulo ${\Psi}^{0}$}
$$ 
which is not subelliptic near $\set{{\tau}= 0}$, since $D_t$ is not by Example~\ref{scalarex}.
\end{exe}

\begin{exe}\label{ex2}
Let $P =  h D_t\id_2 + i F(t)|D_x|$ where
\begin{equation*}
 F(t) =
\begin{pmatrix}
t^2 + t^8 & t^3 - t^7 \\ t^3 - t^7 & t^4 +t^6 
\end{pmatrix} 
 = 
\begin{pmatrix}
 1 &  t \\ -t & 1
\end{pmatrix}
\begin{pmatrix}
t^2 & 0 \\ 0 & t^6
\end{pmatrix}
\begin{pmatrix}
 1 &  -t \\ t & 1
\end{pmatrix}.
\end{equation*}
Then we have
\begin{equation*}
 P = (1 + t^2)^{-1}
\begin{pmatrix}
 1 &  t \\ -t & 1
\end{pmatrix}
\begin{pmatrix}
D_t + i(t^2+ t^4)|D_x| & 0 \\ 0 & D_t + i(t^6 + t^8)|D_x|
\end{pmatrix}
\begin{pmatrix}
 1 &  -t \\ t & 1
\end{pmatrix}
\end{equation*}
modulo ${\Psi}^{0}$, which is subelliptic near $\set{{\tau}= 0}$ with a loss
of $6/7$ derivatives  by Example~\ref{scalarex}.
This operator is, element for element, a
higher order perturbation of the operator of Example~\ref{ex1}.
\end{exe}

The problem is that condition~\eqref{noninvcond} in {\em not} invariant in the systems
case. Instead, we shall consider the following invariant generalization of~\eqref{scalarcond}.

\begin{defn} 
Let  $0 \le F(t)\in L^\infty_{loc}(\br)$ be an $N \times N$ system, then we define
\begin{equation} \label{omegadef}
{\Omega}_{\delta}(F) = \set{t: \min_{\mn u = 1} \w{F(t)u,u} \le
  {\delta}} \qquad {\delta} > 0
\end{equation}
which is well-defined almost everywhere and contains $|F|^{-1}(0)$.
\end{defn}

Observe that one may also use this definition in the scalar
case, then ${\Omega}_{\delta}(f) = f^{-1}([0,{\delta}])$ for
non-negative functions ~$f$.

\begin{rem}\label{subinvrem}
Observe that if $F \ge 0$ and $E$ is invertible then
we find that 
\begin{equation}\label{subinv}
{\Omega}_{\delta}(E^*FE)\subseteq {\Omega}_{C\delta}(F)  
\end{equation}
where $C = \mn {E^{-1}}^2$. 
\end{rem}

\begin{exe}
For the matrix $F(t)$ in Example~\ref{ex2} we find that
$|{\Omega}_{\delta}(F)| \le C {\delta}^{1/6}$ for $0 < {\delta} \le 1$, and
for the  matrix in Example~\ref{ex1} we find that
$|{\Omega}_{\delta}(F)| = \infty$, $\forall\, {\delta}$.
\end{exe}

We also have examples when the semidefinite imaginary part vanishes of infinite order.

\begin{exe}
Let 
$
 0 \le f(t,x) \le C e^{-1/|t|^{{\sigma}}}
$, ${\sigma} > 0$,
then we obtain that 
\begin{equation*}
 |{\Omega}_{\delta}(f_{x})|\le C_0 |\log {\delta}|^{-1/{\sigma}} \qquad
 \forall\, {\delta} >0 \quad \forall\, x
\end{equation*}
where $f_x(t) = f(t,x)$.
(We owe this example to Y. Morimoto.)
\end{exe}

We shall study systems where the imaginary part $F$ vanishes of finite
order, so that $ |{\Omega}_{\delta}(F)| \le C {\delta}^{{\mu}}$ for
${\mu} > 0$. In general, the largest exponent could be any ${\mu}> 0$, for
example when $F(t) = |t|^{1/{\mu}}\id_N$.
But for $C^\infty$ systems the best exponent is ${\mu} = 1/k$ for an
even ~$k$, by the following  result, which is Proposition~A.2  in
~\cite{de:pseudospec}.

\begin{rem}\label{subcondlem}
Assume that $0 \le F(t) \in C^\infty(\br)$ is an $N \times N$ system such that
$F(t) \ge c > 0$ when $|t| \gg 1$.
Then we find that 
\begin{equation*}
 |{\Omega}_{\delta}(F)| \le C {\delta}^{{\mu}} \qquad 0 <
 {\delta} \le 1
\end{equation*}
if and only if ${\mu} \le 1/k$ for an even $k \ge 0$ so that
\begin{equation}\label{dersubcond}
 \sum_{j \le k}|\partial_t^j\w{F(t)u(t), u(t)}|/\mn{u(t)}^2  > 0 \qquad
 \forall\,t
\end{equation}
for any $0 \ne u(t) \in C^\infty(\br)$.
\end{rem}

\begin{exe} 
For the scalar symbols ${\tau} + i f(t, x,{\xi})$ in
Example~\ref{scalarex} we find from Remark~\ref{subcondlem}
that \eqref{scalarcond}  
is equivalent to
\begin{equation*}
 |\set{t: f(t,x,{\xi}) \le {\delta}}| = |{\Omega}_{\delta}(f_{x,{\xi}})|\le C{\delta}^{1/k}\qquad
 0 < {\delta} \le 1  \qquad |{\xi}| = 1
\end{equation*}
where $f_{x,{\xi}}(t) = f(t,x,{\xi})$.
\end{exe}

The following example shows that for subelliptic type of estimates it is not sufficient to have
conditions only on the vanishing of the symbol, we also need conditions 
on the semibicharacteristics of the eigenvalues.

\begin{exe}\label{subex}
Let 
\begin{equation*}
 P =  D_t\id_2 +
{\alpha}  \begin{pmatrix}
  D_x & 0 \\ 0 & -D_x
\end{pmatrix}
+ i(t - {\beta} x)^2|D_x|\id_2 \qquad (t,x)\in \br^2
\end{equation*}
with ${\alpha}$, ${\beta} \in \br$, then we see from the scalar case in Example~\ref{scalarex}
that $P$ is subelliptic near $\set{{\tau}= 0}$ with a loss of $2/3$
derivatives if and only either ${\alpha} = 0$  
or ${\alpha} \ne 0$ and ${\beta} \ne \pm 1/{\alpha}$.
\end{exe}

\begin{defn}\label{apprdef}
Let  $Q \in C^\infty(T^*X)$ be an  $N \times N$ system and let $w_0
\in {\Sigma} \subset T^*X$, then $Q$
satisfies the {\em approximation property} on~${\Sigma}$ near~$w_0$ if 
there exists a $Q$ invariant $C^\infty$ subbundle $\Cal V$ of ~$\bc^N$
over $T^*X$ such that $\Cal V(w_0) = \ker Q^N(w_0)$ and
\begin{equation}\label{approxcond}
 \re\w{Q(w)v,v} = 0\qquad v \in \Cal V(w) \qquad w \in {\Sigma} 
\end{equation}
near $w_0$. That  $\Cal V$ is  $Q$ invariant means that $Q(w)v \in \Cal V(w)$ for
$ v \in \Cal V(w)$.  
\end{defn} 

Here $\ker Q^N(w_0)$ is  the
space of the generalized eigenvectors corresponding to the zero eigenvalue.
The symbol of the system in Example~\ref{subex} satisfies the approximation
property on ${\Sigma} = \set{{\tau} = 0}$ if and only if ${\alpha}
= 0$. 

Let $\wt Q = Q\restr{\Cal V}$ then $\im i \wt Q = \re \wt
Q= 0$ so Lemma~\ref{semiprop} gives that $\ran \wt Q \bot \ker \wt Q$ on
~${\Sigma}$. Thus $\ker \wt Q^N = 
\ker \wt Q$ on ~${\Sigma}$, and since $\ker \wt Q^N(w_0) = \Cal V(w_0)$ we find that
$ \ker Q^N(w_0) = \Cal V(w_0)= \ker Q(w_0)$.

\begin{rem} \label{condrem}
Assume that $Q$ satisfies the approximation property on~ the
$C^\infty$ hypersurface ${\Sigma}$
and is quasi-symmetric with 
respect to  $V  \notin T{\Sigma}$.
Then the limits of the non-trivial semibicharacteristics
of the eigenvalues of~$Q$ close to zero coincide with the bicharacteristics
of ~${\Sigma}$. 
\end{rem}

In fact, the approximation
property in Definition~\ref{apprdef} gives that $ \w{\re Qu,u
} = 0$ for $u\in \ker Q$ when ${\tau}= 0$. Since $\im
Q \ge 0$ we find that
\begin{equation}\label{tangvan}
 \w{dQu,u } = 0 \qquad \forall\, u
 \in \ker Q \qquad\text{on $T{\Sigma}$}
\end{equation}
By Remark~\ref{qshamrem} the limits of the  non-trivial
semibicharacteristics of the eigenvalues close to zero  of 
$Q$ are curves with tangents determined by
$\w{dQ u,u}$ for $u \in \ker Q$.
Since $V \re Q \ne 0$ on $\ker Q$ we find from~\eqref{tangvan}
that the limit curves coincide with the bicharacteristics
of ~${\Sigma}$, which are the flow-outs of the Hamilton vector field.

\begin{exe} 
Observe that Definition~\ref{apprdef} is empty if $\dim \ker Q^N(w_0) = 0$.
If $\dim \ker Q^N(w_0) > 0$, then there exists ${\varepsilon} > 0$  and a neigborhood
${\omega}$ to $w_0$ so that
\begin{equation}\label{spprojdef}
 {\Pi}(w) = \frac{1}{2{\pi}i}\int_{|z|= {\varepsilon}} (z\id_N - Q(w))^{-1}\,dz
 \in C^\infty({\omega})
\end{equation}
is the spectral projection on the (generalized) eigenvectors with
eigenvalues having absolute value less than~${\varepsilon}$. Then $\ran {\Pi}$ is a $Q$
invariant bundle over~${\omega}$ so that $\ran {\Pi}(w_0)  = \ker
Q^N(w_0)$. Condition~\eqref{approxcond} with $\Cal V = 
\ran {\Pi}$ means that ${\Pi}^*\re Q{\Pi} \equiv 0$ in
~${\omega}$. When $\im Q(w_0) 
\ge 0 $ we find that ${\Pi}^*Q{\Pi}(w_0) = 0$, then ~$Q$ satisfies the
approximation property on ~${\Sigma}$ near ~$w_0$
with $\Cal V = \ran {\Pi}$ if and only if 
$$
d({\Pi}^*(\re Q){\Pi})\restr{T{\Sigma}} \equiv 0 \qquad \text{ near ~$w_0$}
$$
\end{exe}

\begin{exe}\label{normformex}
If $Q$ satisfies the approximation property on ${\Sigma}$, then by choosing an
orthonormal basis for $\Cal V$ and extending it to an orthonormal basis for
~$\bc^N$ we obtain the system on the form
\begin{equation*}\label{normformex0}
 Q = 
\begin{pmatrix}
Q_{11} & Q_{12} \\ 0 & Q_{22} 
\end{pmatrix}
\end{equation*}
where $Q_{11}$ is $K \times K$ system such that $Q^{N}_{11}(w_0) = 0$,
$\re Q_{11} = 0$ on ${\Sigma}$ and $|Q_{22}| \ne 0$.
By multiplying from left with
\begin{equation*}
 \begin{pmatrix}
\id_K & -Q_{12}Q_{22}^{-1} \\ 0 & \id_{N-K} 
\end{pmatrix}
\end{equation*}
we obtain that $Q_{12} \equiv 0$ without changing $Q_{11}$ or $Q_{22}$. 
\end{exe}

In fact, the eigenvalues of ~$Q$ are then eigenvalues of either
~$Q_{11}$ or ~$Q_{22}$. Since ~$\Cal V(w_0)$ are the (generalized)
eigenvectors corresponding to the zero eigenvalue of~$Q(w_0)$ we find
that all eigenvalues of ~ $Q_{22}(w_0)$ are non-vanishing, thus
$Q_{22}$ is invertible near~$w_0$,

\begin{rem}
If $Q$ satisfies the approximation property on ${\Sigma}$ near ~~$w_0$,
then it satisfies the approximation property on ~ ${\Sigma}$ near~$w_1$,
for $w_1$ sufficiently close to ~$w_0$.  
\end{rem}

In fact, let $Q_{11}$ be the restriction of $Q$ to $\Cal V$ as in Example~~\ref{normformex},
then since $\re Q_{11} = \im i Q_{11} = 0$ on ${\Sigma}$ we find from Lemma~\ref{semiprop} that
 $\ran Q_{11} \bot \ker Q_{11}$ and
$\ker Q_{11} = \ker Q_{11}^N$ on
${\Sigma}$. Since $Q_{22}$ is invertible in ~\eqref{normformex0}, we
find that $\ker Q \subseteq \Cal V$. Thus, by using the spectral
projection ~\eqref{spprojdef} of $Q_{11}$ near ~$w_1 \in {\Sigma}$
for small enough~${\varepsilon}$ we obtain an $Q$ invariant subbundle
$\wt {\Cal V} \subseteq \Cal V$ so that $\wt {\Cal V}(w_1) = \ker Q_{11}(w_1)
= \ker Q^N(w_1)$.

If  $Q \in C^\infty$ satisfies the approximation property and $Q_E =
E^*QE$ with invertible $E\in C^\infty$, then it follows from the proof of 
Proposition~\ref{fininv} below that there exist invertible $A$, $B \in C^\infty$
so that $AQ_E$ and  $Q^*B$   satisfy the approximation property.

\begin{defn}\label{subdef}
Let $P(w) \in C^\infty (T^*X) $ be an $N \times N$ system and ${\mu}
\in \br_+ $. Then ~$P$ is 
of {\em  finite type~ ${\mu}$} at~ $w_0 \in T^*X$ 
if there exists a neighborhood
${\omega}$ of $w_0$, a $C^\infty$ hypersurface~${\Sigma}\ni w_0$, a real
$C^\infty$ vector field $V \notin T{\Sigma}$
and an invertible symmetrizer ~$M \in C^\infty$ so that 
$Q = MP$ is quasi-symmetric with respect to~$V$ in ${\omega}$ and
satisfies the approximation property on ${\Sigma}\bigcap {\omega}$. Also,  
for every bicharacteristic ~${\gamma}$ of ~${\Sigma}$ 
the arc length
\begin{equation}\label{subcond}
\big| {\gamma}\cap {\Omega}_{\delta}(\im Q) \cap
{\omega}\big| \le C{\delta}^{\mu}\qquad 0 < {\delta} \le 1
\end{equation}
The operator $P \in {\Psi}^m_{cl}$ is of finite
type  ${\mu}$ at~ $w_0$ if the principal symbol ~${\sigma}(P)$ is
of finite type when ~$|{\xi}| = 1$.
\end{defn}

Recall that the bicharacteristics of a hypersurface in $T^*X$ are the
flow-outs of the Hamilton vector field of~${\Sigma}$.
Of course, if $P$ is elliptic then it is trivially of
finite type $0$, just choose $M = iP^{-1}$ to obtain $Q = i\id_N$.
If $P$ is of finite type, then it is quasi-symmetrizable by
definition and thus of principal type.

\begin{rem}\label{nfrem}
Observe that since  $0 \le \im Q \in C^\infty$ we obtain from
Remark~\ref{subcondlem} that the largest exponent in ~\eqref{subcond} is
${\mu} = 1/k$ for an even ~$k \ge 0$. Also, we may assume that 
\begin{equation}\label{XQC}
 \im\w{Qu,u} \ge c\mn{Qu}^2 \qquad \forall \, u \in \bc^N
\end{equation}
\end{rem}

In fact, by adding $i{\varrho}P^*$ to $M$ we obtain~\eqref{XQC} for
large enough~${\varrho}$ by~\eqref{immp}, and this does not change $\re Q$.

\begin{exe}\label{Qex}
Assume that $Q$ is quasi-symmetric with respect to the real vector
field $V$, satisfying~\eqref{subcond} and
the approximation property on ${\Sigma}$. Then by choosing an
orthonormal basis and changing the symmetrizer as in
Example~\ref{normformex} we obtain the system on the form 
\begin{equation*}
 Q = 
\begin{pmatrix}
Q_{11} & 0 \\ 0 & Q_{22} 
\end{pmatrix}
\end{equation*}
where $Q_{11}$ is $K \times K$ system such that $Q^{N}_{11}(w_0) = 0$,
$\re Q_{11} = 0$ on ${\Sigma}$ and $|Q_{22}| \ne 0$. Since $Q$ is
quasi-symmetric with respect to ~$V$ we also obtain that  $Q_{11}(w_0) = 0$, $\re VQ_{11}
> 0$, $\im Q \ge 0$ and $Q$ satisfies~\eqref{subcond}. In fact, then we find from
Lemma~\ref{semiprop} that $\im Q \bot \ker Q$ which gives $\ker Q^N =
\ker Q$. Note that ${\Omega}_{\delta}(\im Q_{11}) \subseteq
{\Omega}_{\delta}(\im Q)$, so
$Q_{11}$ satisfies~\eqref{subcond}. 
\end{exe}

\begin{exe} \label{scalsubex} 
In the scalar case, we find from Example~\ref{scalarcase} that $p\in C^\infty(T^*X)$ is
quasi-symmetrizable with respect to $H_t = \partial_{\tau}$ if and only if  
\begin{equation}\label{pex}
 p(t,x;{\tau},{\xi}) = q(t,x;{\tau},{\xi})({\tau} + i
 f(t,x,{\xi})) 
\end{equation}
with $f \ge 0$ and $q \ne 0$. If $f(t,x,{\xi}) \ge c > 0$ when $|(t,x,{\xi})| \gg
1$ we find by taking $q^{-1}$ as symmetrizer 
that $p$ is of finite type ${\mu}$ if and only
if ${\mu} = 1/k$ for an even ~$k$ such that
\begin{equation*}
 \sum_{j \le k}|\partial_{t}^kf(t,x,{\xi}) | > 0 \qquad \forall\, x \
{\xi} 
\end{equation*}
by  Remark~\ref{subcondlem}. In fact, the approximation property
on ${\Sigma} = \set{{\tau} = 0}$ is trivial since $f$ is real.
\end{exe}

\begin{prop}\label{fininv}
If $P(w)\in C^\infty(T^*X)$ is of finite type ${\mu}$ at~$w$ then $P^*$ is of finite
type ${\mu}$ at~$w$. If $A(w)$ and $B(w)\in C^\infty(T^*X)$ are invertible, then
$APB$ is of finite type ${\mu}$ at~$w$.
\end{prop}

\begin{proof}
Let $M$ be the symmetrizer in Definition~\ref{subdef} so that $Q = MP$
is quasi-symmetric with respect to $V$. By choosing a suitable basis 
and changing the symmetrizer as
in Example~\ref{Qex}, we may write
\begin{equation}\label{specformorig}
 Q = 
\begin{pmatrix}
Q_{11} & 0 \\ 0 & Q_{22} 
\end{pmatrix}
\end{equation}
where $Q_{11}$ is $K \times K$ system such that $Q_{11}(w_0) = 0$,  $V
\re Q_{11} > 0$, $\re Q_{11} = 0$ on
${\Sigma}$ and $Q_{22}$ is invertible. We also have $\im Q \ge
0$ and $Q$ satisfies~\eqref{subcond}. 
Let $\Cal V_1 = \set{u\in \bc^N:\ u_j = 0 \text{ for $j > K$}}$ and $\Cal V_2 =
\set{u\in \bc^N:\ u_j = 0 \text{ for $j \le K$}}$, these are $Q$
invariant bundles such that $ \Cal V_1 \oplus \Cal V_2 = \bc^N$.

First we are going to show that $\wt P =APB$ is of finite type. 
By taking $\wt M = B^{-1}MA^{-1}$ we find that 
\begin{equation}\label{specform}
\wt M \wt P = \wt Q  = B^{-1}QB
\end{equation}
and it is clear that $B^{-1}\Cal V_j$ are $\wt Q$ invariant
bundles, $j = 1$, $2$. By choosing bases in  $B^{-1}\Cal V_j$ for
$j = 1$, $2$, we obtain a basis for $\bc^N$ in which $\wt Q$ has a block
form:
\begin{equation}\label{qblock}
 \wt Q = 
\begin{pmatrix}
\wt Q_{11} & 0 \\ 0 & \wt Q_{22}  
\end{pmatrix}
\end{equation}
Here $\wt Q_{jj}: B^{-1}\Cal V_j \mapsto B^{-1}\Cal V_j$, is given by
$\wt Q_{jj} = B_{j}^{-1} Q_{jj}B_{j}$ with 
$$B_{j}: B^{-1} \Cal V_j\ni u \mapsto Bu \in \Cal V_j \qquad j= 1,\ 2$$ 
By multiplying $\wt Q$ from the left with 
\begin{equation*}
\Cal B =
\begin{pmatrix}
 B_{1}^*B_{1} & 0 \\ 0 &  B_{2}^*B_{2}
\end{pmatrix}
\end{equation*}
we obtain that 
\begin{equation*}
\ol Q = \Cal B \wt Q = \Cal B \wt M \wt P =
\begin{pmatrix}
 B_{1}^*Q_{11}B_{1} & 0 \\ 0 &  B_{2}^*Q_{22}B_{2}
\end{pmatrix}
=
\begin{pmatrix}
\ol Q_{11} & 0 \\ 0 & \ol Q_{22}  
\end{pmatrix}
\end{equation*}
It is clear that $\im \ol Q \ge 0$, $Q_{11}(w_0) = 0$, $\re \ol Q_{11}
= 0$ on ${\Sigma}$, $|\ol Q_{22}| \ne 0$
and $V \re \ol Q_{11} > 0$ by Proposition~\ref{invrem0}. Finally, we
obtain from Remark~\ref{subinvrem} that 
\begin{equation}\label{subinv0}
{\Omega}_{\delta}(\im \ol Q) \subseteq {\Omega}_{C\delta}(\im Q)
\end{equation}
for some $C > 0$, which proves that $\wt P = APB$ is of finite type.
Observe that $\ol Q = AQ_B$, where $Q_B = B^*QB$ and $A = \Cal B B^{-1}(B^*)^{-1}$.

To show that $P^*$ also is of finite type, we may assume as before that
$Q= MP$ is on the form~\eqref{specformorig} with $Q_{11}(w_0) = 0$,  $
V \re Q_{11} > 0$, $\re Q_{11} = 0$ on
${\Sigma}$, $Q_{22}$ is invertible, $\im Q \ge 0$ and  $Q$ satisfies~\eqref{subcond}.
Then we find that
\begin{equation*}
 -P^*M^* = -Q^* = 
\begin{pmatrix}
 - Q_{11}^* & 0 \\ 0 & - Q_{22}^*   
\end{pmatrix}
\end{equation*}
satisfies the same conditions with respect to $-V$, so it is of finite
type with multiplier~ $\id_N$. By the first part of the proof we 
obtain that $P^*$ is of finite type, which finishes the proof.
\end{proof}

\begin{thm}\label{subthm}
Assume that  $P \in {\Psi}^m_{cl}(X)$ is an $N \times N$ system of
finite type ${\mu} > 0$ near~$w_0 \in T^*X\setminus 0$, then $P$ is subelliptic at ~$w_0$
with a loss of $1/{\mu}+1$ derivatives:
\begin{equation}\label{subprop}
 Pu\in H_{(s)} \text{ at $w_0$ } \implies u
 \in H_{(s+m- 1/{\mu}+1)}  \text{ at $w_0$ }
\end{equation}
for $ u \in \Cal D'(X, \bc^N)$.
\end{thm}

Observe that the largest exponent is ${\mu} = 1/k$ for an even ~$k$ by
Remark~\ref{nfrem}, and then $1/{\mu}+1 = k/k+1$. Thus Theorem~\ref{subthm} generalizes
Proposition~27.3.1 in ~\cite{ho:yellow} by Example~\ref{scalsubex}.

\begin{exe}\label{simplex}
Let  
\begin{equation*}
 P(t,x;{\tau},{\xi}) = {\tau}M(t,x,{\xi})+ iF(t,x,{\xi}) \in
 S^1_{cl}
\end{equation*}
where  $M \ge c_0  > 0$ and $F \ge 0$ satisfies 
\begin{equation}\label{subcond0}
\left|\set{t:\ \inf_{|u| = 1}\w{F(t,x,{\xi})u,u} \le {\delta}} \right| \le C
{\delta}^{\mu} \qquad |{\xi}|=1
\end{equation}
for some ${\mu} > 0$. 
Then $P$ is quasi-symmetrizable with respect to $\partial
_{{\tau}}$ with symmetrizer $\id_N$. When ${\tau} =
  0$ we obtain that $\re P = 0$, so by taking $\Cal V = \ran {\Pi}$
  for the spectral projection
  ${\Pi}$ given by ~\eqref{spprojdef} for $F$, we find that
$P$ satisfies the approximation
  property with respect to ${\Sigma}= \set{{\tau} = 0}$. 
Since ${\Omega}_{\delta}(\im P) =
{\Omega}_{\delta}(F)$ we find from ~\eqref{subcond0}
that $P$ is of finite type ~${\mu}$.
Observe that if $F(t,x,{\xi}) \ge c > 0$ when $|(t,x,{\xi})| \gg 1$ we find from
Remark~\ref{subcondlem} that ~\eqref{subcond0} is satisfied if
and only if ${\mu} \le 1/k$ for an even $k \ge 0$ so that 
\begin{equation*}
 \sum_{j \le k}|\partial_t^j\w{F(t,x,{\xi})u(t), u(t)}| > 0 \qquad
 \forall\,t, x, {\xi}
\end{equation*}
for any $0 \ne u(t) \in C^\infty(\br)$. Theorem~\ref{subthm}
gives that $P(t,x,D_t,D_x)$ is subelliptic near~$\set{{\tau}=0}$ with a loss of $k/k+1$
derivatives. 
\end{exe}

\begin{proof}[Proof of Theorem~\ref{subthm}] 
First, we may reduce to the case $m = s = 0$ by replacing $u$ and~$P$ by $\w{D}^{s+m}u$
and $\w{D}^{s}P\w{D}^{-s-m} \in {\Psi}^0_{cl}$. Now $u \in H_{(-K)}$ for
some ~$K$ near~$w_0$, and 
it is no restriction to assume $K = 1$. In fact, if $K > 1$ then by using that $Pu \in
H_{(1-K)}$ near~$w_0$, we obtain that $u \in H_{(-K + {\mu}/{\mu}+1)}$  near~$w_0$
and we may iterate this argument until  $u \in H_{(-1)}$   near~$w_0$.
By cutting off with ${\phi}\in S^0_{1,0}$ we may assume that $v = {\phi}(x, D)u \in
H_{(-1)}$ and $Pv = [P, {\phi}(x, D)]u + {\phi}(x, D)Pu \in  H_{(0)}$
since $[P, {\phi}(x, D)] \in {\Psi}^{-1}$. If ${\phi} \ne 0$ in a conical
neighborhood of ~~$w_0$ it suffices to prove that $v \in  H_{(-1/{\mu}+1)}$.

By Definition~\ref{subdef} and Remark~\ref{nfrem} there exist
a $C^\infty$ hypersurface ${\Sigma}$, a real $C^\infty$  vector field
$V \notin T{\Sigma}$, an
invertible symmetrizer $M \in C^\infty$ so that
$Q = MP$ satisfies~\eqref{subcond}, the approximation property
on~${\Sigma}$, and
\begin{align}
&V\re Q  \ge c -  C \im Q \qquad c > 0 \label{sub1}\\
&\im Q \ge c\, Q^*Q \label{sub2}
\end{align}
in a neighborhood ${\omega}$ of $w_0$. By extending by
homogeneity, we can assume that $V$, $M$ and ~$Q$ are homogeneous of degree~$0$.

Since ~\eqref{sub1} is stable under small perturbations in $V$ we can
replace $V$ with $H_t$ for some real $t \in C^\infty$.
By solving the initial value problem $H_t {\tau} \equiv -1$,
${\tau}\restr{\Sigma} = 0$, and completing to a symplectic  $C^\infty$
coordinate system $(t,{\tau},x,{\xi})$,
we obtain that $\Sigma = \set{\tau = 0}$ in a
neighborhood of ~$w_0 = (0,0,x_0,{\xi}_0)$, ${\xi}_0 \ne 0$.
We obtain from Definition~\ref{apprdef} that 
\begin{equation} \label{sub3}
\re \w{Qu,u}  = 0
\qquad \text{when $u \in \Cal V$ and ${\tau}=0$}
\end{equation}
near $w_0$. Here $\Cal V$ is a $Q$ invariant $C^\infty$ subbundle of ~$\bc^N$
such that $\Cal V(w_0) = \ker Q^N(w_0)  = \ker Q(w_0)$ by
Lemma~\ref{semiprop}. By condition~\eqref{subcond} we have that
\begin{equation}\label{sub4}
\big| {\Omega}_{\delta}(\im Q_{x,{\xi}}) \cap
\set{|t| < c}\big| \le C {\delta}^{\mu} 
\end{equation}
when $ |(x,{\xi})- (x_0,{\xi}_0)| < c$, here $Q_{x,{\xi}}(t) = Q(t,0,x,{\xi})$.

Next, we shall localize the estimate.
Choose $\set{{\varphi}_j}_j \in S^0_{1,0}$
and  $\set{{\psi}_j}_j \in S^0_{1,0}$ with values in $\ell^2$, such
that ${\varphi}_j \ge 0$, ${\psi}_j \ge 0$, $\sum_{j} {\varphi}_j^2 =
1$, ${\psi}_j{\varphi}_j \equiv {\varphi}_j$ and ${\psi}_j$ is
supported where $|({\tau},{\xi})| \cong 2^j$. Since these are Fourier
multipliers we find that
$
 \sum_{j}^{}{\varphi}_j(D_{t,x})^2 = 1 
$
and 
\begin{equation*}
 \mn u^{2}_{(s)} \cong \sum_{j} 2^{2sj}\mn{{\varphi}_j(D_{t,x})u}^2 \qquad u
 \in \Cal S
\end{equation*}
Let $Q_j = {\psi}_j Q$ be the
localized symbol, and let $h_j = 2^{-j} \le 1$. Since $Q_j \in
S^0_{1,0}$ is supported where $|({\tau},{\xi})| \cong 2^j$, we find that
$Q_j(t,x,{\tau},{\xi}) = \wt Q_j(t,x,h_j{\tau},h_j{\xi})$ where $\wt Q_j \in
C^\infty_0(T^*\br^n)$ uniformly.
We shall obtain Theorem~\ref{subthm} from the following result, which
is Proposition~6.1 in \cite{de:pseudospec}.

\begin{prop}\label{qestprop}
Assume that $Q \in C_\rb^\infty(T^*\br^n)$ is an $N \times N$ system satisfying
~\eqref{sub1}--\eqref{sub4} in a neighborhood of\/ ~$w_0 =
(0,0,x_0,{\xi}_0)$ with $V = \partial_{\tau}$ and
${\mu} > 0$. Then there exists~$h_0 > 0$ 
and $R \in  C_\rb^\infty(T^*\br^n)$ so that $w_0 \notin \supp R$ and
\begin{equation}\label{locsubest}
  h^{1/{\mu}+1} \mn {u} \le C (\mn{Q(t,x,hD_{t,x}) u} +
  \mn{R^w(t,x,hD_{t,x})u} + h\mn u) \qquad 0 < h \le h_0 
\end{equation}
for any $u \in C^\infty_0(\br^n, \bc^N)$.
\end{prop}

Here $C_\rb^\infty$ are $C^\infty$ functions with
$L^\infty$ bounds on any derivative, and the result is uniform in the
usual sense.
Observe that this estimate can be extended to a semiglobal estimate.
In fact, let ${\omega}$ be a neighborhood of $w_0$ such that $\supp R \bigcap 
{\omega} = \emptyset$, where $R$ is given by Proposition~\ref{qestprop}.
Take ${\varphi}\in C^\infty_0({\omega})$ such that $0 \le {\varphi} \le
1$ and  ${\varphi} = 1$ in a neighborhood of ~$w_0$. 
By substituting ${\varphi}(t,x,hD_{t,x})u$
in~\eqref{locsubest} we obtain from the calculus 
\begin{equation}\label{locsubest0}
  h^{1/{\mu}+1}  \mn {{\varphi}(t,x,hD_{t,x})u} \le C_N
 (\mn{{\varphi}(t,x,hD_{t,x}) Q(t,x,hD_{t,x})u} + h\mn u) \qquad \forall\, u \in C^\infty_0
\end{equation}
for small enough ~$h$ since $R{\varphi} \equiv 0$ and $
\mn{[Q(t,x,hD_{t,x}),{\varphi}(t,x,hD_{t,x})]u} \le C h\mn u$. Thus, if $Q$
satisfies conditions~\eqref{sub1}--\eqref{sub4} near any $w \in K
\Subset T^*\br^n$, then by using Bolzano-Weierstrass we obtain the
estimate~\eqref{locsubest} with $\supp R \bigcap K = \emptyset$.

Now, by using that $\wt Q_j$ satisfies~\eqref{sub1}--\eqref{sub4} in a
neighborhood of $\supp {\varphi}_j$, we obtain the
estimate~\eqref{locsubest} for $\wt Q_j(t,x,hD_{t,x})$ with $h = h_j = 2^{-j}
\ll 1$ and
$R = R_j \in S^{0}_{1,0}$
such that $\supp {\varphi}_j \bigcap \supp R_j = \emptyset$.  Substituting
${\varphi}_j(D_{t,x}) u$ we obtain for $j \gg 1$ that
\begin{equation*}
  2^{-j/{\mu}+1}  \mn {{\varphi}_j(D_{t,x}) u} \le C_N
 (\mn{Q_j(t,x,D_{t,x}){\varphi}_j(D_{t,x})u} + \mn{\wt R_ju} +
 2^{-j}\mn{{\varphi}_j(D_{t,x})u}) \quad \forall\, u \in \Cal S'
\end{equation*}
where $\wt R_j = R_j(t,x,D_{t,x}){\varphi}_j(D_{t,x}) \in {\Psi}^{-N}$ with values in
$\ell^2$. Now since $Q_j$ and $Q$ are uniformly  bounded in
$S^0_{1,0}$ the calculus gives that 
$$Q_j(t,x,D_{t,x}){\varphi}_j(D_{t,x}) =
{\varphi}_j(D_{t,x})Q(t,x,D_{t,x}) + {\varrho}_j(t,x,D_{t,x})$$ 
where $\set{{\varrho}_j}_j \in
{\Psi}^{-1}$ with values in ~~$\ell^2$.  
Thus, by squaring and summing up, we
obtain by continuity that
\begin{equation}\label{hest}
 \mn{u}_{(-1/{\mu}+1)}^2 \le C(\mn{Q(t,x,D_{t,x}) u}^2 + \mn{u}_{(-1)}^2)\qquad u
 \in H_{(-1)}
\end{equation}
Since $Q(t,x,D_{t,x})= M(t,x,D_{t,x})P(t,x,D_{t,x})$ modulo ${\Psi}^{-1}$ where $M \in
{\Psi}^0$, the calculus gives
\begin{multline}\label{simpest}
 \mn {Q(t,x,D_{t,x})u}  \le C( \mn{M(t,x,D_{t,x})P(t,x,D_{t,x})u} + \mn u_{(-1)}) \\\le
 C'(\mn{P(t,x,D_{t,x})u} + \mn u_{(-1)}) \quad u  \in H_{(-1)} 
\end{multline}
which together with ~~\eqref{hest} proves Theorem~\ref{subthm}.
\end{proof}

\bibliographystyle{amsplain}

\providecommand{\bysame}{\leavevmode\hbox to3em{\hrulefill}\thinspace}
\providecommand{\MR}{\relax\ifhmode\unskip\space\fi MR }
\providecommand{\MRhref}[2]{%
  \href{http://www.ams.org/mathscinet-getitem?mr=#1}{#2}
}
\providecommand{\href}[2]{#2}

\end{document}